\documentclass[graybox]{svmult}

\usepackage{type1cm}   
\usepackage{makeidx}   
\usepackage{graphicx}   
\usepackage[bottom]{footmisc}   
\usepackage{newtxtext}        
\usepackage{newtxmath}   
\usepackage{amsmath,amsfonts}
\usepackage{titlesec}
\usepackage{booktabs}
\usepackage{multirow}
\usepackage{bbm}
\usepackage{todonotes}
\usepackage{hyperref}
\usepackage{cite}

\def\R{{\mathbb R}}
\def\E{{\mathbb E}}

\newtheorem{assumption}[theorem]{Assumption}

\newcommand{\Rp}{\mathbb{R}_{>0}} 

\newcommand{\defm}[1]{\emph{#1}}

\newcommand{\Lcal}{\ensuremath{\mathcal{L}}}
\newcommand{\de}{\mathrm{d}}
\newcommand{\N}{\ensuremath{\mathbb{N}}}

\makeindex

\begin{document}


\title*{Controllability of continuous networks and a kernel-based learning approximation}
\titlerunning{Controllability of continuous networks and a kernel-based learning approximation}
\author{Michael Herty, Chiara Segala, Giuseppe Visconti}
\authorrunning{M. Herty, C. Segala, G. Visconti}
\institute{
Michael Herty, Chiara Segala \at  IGPM, RWTH Aachen University, Templergraben 55, D-52062 Aachen, Germany; \email{herty,segala@igpm.rwth-aachen.de}
\and  Giuseppe Visconti \at Department of Mathematics ``G.~Castelnuovo'', Sapienza University of Rome, P.le Aldo Moro 5, 00185 Roma, Italy; \email{giuseppe.visconti@uniroma1.it}
}
	
\maketitle
	
\let\thefootnote\relax\footnotetext{\hspace{1.8ex}
M.H. and C.S. thank the Deutsche Forschungsgemeinschaft (DFG, German Research Foundation) for the financial support 
442047500/SFB1481 within the projects B04, B05, B06, and SPP 2298 Theoretical Foundations of Deep Learning  within the Project(s) HE5386/23-1, Meanfield Theorie zur Analysis von Deep Learning Methoden (462234017).

G.V. acknowledges the support of MUR (Ministry of University and Research) under the MUR-PRIN PNRR Project 2022 No.~P2022JC95T ``Data-driven discovery and control of multi-scale interacting artificial agent systems'', the PRIN Project 2017 No.~2017KKJP4X ``Innovative numerical methods for evolutionary partial differential equations and applications'', the PNRR-MUR project PE0000013-FAIR.
}

\abstract{
Residual deep neural networks  are formulated as interacting particle systems leading to a  description through neural differential equations, and, in the case of large input data, through mean-field neural networks. The mean-field description allows also   the recast of the training processes as  a controllability problem for the solution to the mean-field dynamics. We  show theoretical results on the controllability of the linear microscopic and mean-field dynamics through the Hilbert Uniqueness Method  and  propose a computational approach based on kernel learning methods to solve numerically, and efficiently, the training problem. Further aspects of the structural properties of the mean-field equation will be reviewed.}
\ \\
\textbf{Mathematics Subject Classification (2020)}  49J15, 49J20, 35Q49, 92B20, 90C31.
\\ \\
\textbf{Keywords} Neural networks, mean-field limit, controllability, kernel methods.

\section{{Introduction}} \label{sec:introduction}

In recent years, there has been an increasing interest in machine learning and data science ~\cite{Wooldridge2020,LalmuanawaHussainChhakchhuak2020,MuellerVincentBostrom2016} with applications such as human speech recognition~\cite{MengistuRudzicz2011}, competition in strategic game systems~\cite{LiXingFenghuaCheng2020}, intelligent routing in content delivery networks~\cite{FadlullahZubairMaoTangKato2019}, and autonomous vehicles  operations~\cite{STERN2018}.
The intersection of mathematics and artificial intelligence allows the use of machine learning tools to tackle difficulties arising in numerical methodologies, such as high-dimensional parameter optimization, in the modelling of physics-based operators through experimental data or uncertainty quantification, see e.g. ~\cite{Mishra2019,LyeMishraRay2020,ZhangGuoKarniadakis2020,RaissiPerdikarisKarnidakis2019,DiscacciatiHesthavenRay2020,RayHesthaven2019,MagieraRayHesthavenRohde2020,RayHesthaven2018}.

Here, we are interested in  a particular class of learning-based methods, the deep residual neural networks (ResNets). Given a set of input data $x_i^0$, $i=1,\dots,M$, the ResNet propagates those through the layers $\kappa=0,\dots,L+1$, to provide a state prediction $x_i(L+1)$. This state is compared with given reference data $y_i.$ The dynamics depend on a large set of parameters, called weights $w(\kappa)$ and biases $b(\kappa)$. Their values are obtained as a solution to an optimization problem and the typically iterative process is called  training. The objective or cost is given by a distance $\ell$ between  predictions $x_i(L+1)$  and the reference $y_i.$  ResNets have also been formulated  for infinitely many layers, leading to the definition of neural differential equations~\cite{chen2018neural} and to mean-field neural equations~\cite{HTV_MFResNet}. The continuous formulations are subject to  theoretical investigations and may reduce the computational cost of the training, especially in the case when $M$ is large, see e.g.~\cite{Crevat2019,MeiMontanariNguyen2018,SirignanoSpiliopoulos2020a,SirignanoSpiliopoulos2020b,BaccelliTaillefumier2019,TrimbornGersterVisconti2020,HTTV_training,BonnetCiprianiFornasierHui,Osher:controllability,Zuazua:neuralODE}.

In this work, we take a different point of view and model the training process as a controllability  problem. In very particular cases,  the Hilbert Uniqueness Method (HUM) yields the existence of optimal weights. The HUM is a mathematical technique used in the study of partial differential equations that has been applied in control theory of partial differential equations, see e.g.~\cite{lagnese2006hilbert,limaco2022exact,leiva2008controllability,Coron}. Here, we also show the applicability of suitably formulated training problems for ResNet.

The problem of controllability of continuous neural networks has been  discussed for example in \cite{SONTAG1997177,Osher:controllability,Zuazua:neuralODE}. In~\cite{SONTAG1997177}  the controllability of  continuous-time recurrent neural networks has been established provided that the  activation function is the hyperbolic tangent. This work assumes infinitely many layers $L \to \infty$ but still a finite size samples $M<\infty.$

More recently~\cite{Osher:controllability}, the controllability is discussed for $L$ and $M$ tending to infinity, leading to the mean-field  equation. Therein, its controllability using weights that are piecewise constant (in time) has been established. For further results, we refer to the recent review~\cite{Zuazua:neuralODE}. While the controllability can be established here in the case of linear activation functions, we are also interested in their numerical for  general training tasks. To this end,  we propose an approach based on kernel learning methods. In particular, we approximate the ResNets loss function using kernel based estimation. Those methods are commonly used as powerful machine learning tools, see e.g.~\cite{scholkopf2002learning,williams2006gaussian, kanagawa2018gaussian}, and they are supported by a well-defined theory~\cite{steinwart2008support,shawe2004kernel}.


The paper is organized as follows. In Section~\ref{sec:preliminaries} we introduce deep residual neural networks and discuss  the time-continuous and  the mean-field equations. Section~\ref{sec:control_pb}  provides controllability results for the linear ResNets both from a microscopic and mean-field perspective. In this section, we propose also the numerical approach based on kernel learning methods. In Section~\ref{subsec:kernelmethod} we show numerical experiments  of the microscopic and mean-field neural network.

\section{Preliminaries: Neural differential equations and their mean-field limit}\label{sec:preliminaries}

Assume a deep Residual Neural Network (ResNet) with $L$ hidden layers  and $N_\kappa$  neurons in each layer $\kappa$. The indices $\kappa=0$ and $\kappa=L+1$ denote the input layer and the output layer, respectively.

Consider a data-set $\{x_i^0, y_i\}_{i=1}^M$ of size $M\in\N$, $M\gg 1$. Let each input data $x_i^0 \in\R^d$ and output or target data $y_i \in \R^d$ contain  $d\geq 1$  features or measurements.   The state of the $i$-th input data at the $\kappa$-th layer is denoted by $x_i(\kappa)\in\R^{N_\kappa}$, with $x_i(0)=x_i^0$ and $N_0=d$. The final state $x_i(L+1)\in\R^{N_{L+1}}$ is called {\em output} or prediction.

A ResNet propagates each input data $x_i^0$ according to the deterministic dynamics~\cite{He2015DeepRL}:
\begin{equation} \label{eq:resnet}
	x_i(\kappa+1)=A(\kappa)x_i(\kappa)+\Delta t \, \sigma\big( w(\kappa) x_i(\kappa)+b(\kappa) \big), \quad \kappa=0,\dots,L 
\end{equation}
where $\Delta t>0$ indicates a (pseudo)-time step, $w(\kappa) \in\R^{N_{\kappa+1}\times N_\kappa}$ are the weights and $b(\kappa)\in\R^{N_{\kappa+1}}$ the bias which are the parameters of the network. The matrix $A(\kappa) \in R^{N_{\kappa+1}\times N_\kappa}$ will be specified later and the  function $\sigma: \R\to \R$ is the activation function of the neurons  applied component-wise in~\eqref{eq:resnet}. Examples 
are the identity function $\sigma(x) = x$, the rectified linear unit (ReLU) function $\sigma(x) = \max\{0,x\}$, the sigmoid function $\sigma(x) = \frac{1}{1+\exp(-x)}$, the hyperbolic tangent function $\sigma(x) = \tanh(x)$ or the growing cosine unit (GCU) function $\sigma(x)=x\cos(x)$. 
	
In {\em supervised learning} the parameters $w$ and $b$ are chosen by a {\em training process}
 aiming to minimize the distance $\ell:\R^{N_{L+1}}\to\R^+_0$ between given \emph{targets} $\{y_i\}_{i=1}^M$ and the predictions:
\begin{equation} \label{eq:loss}
	\min_{(w(\cdot),b(\cdot))} \frac{1}{M} \sum_{i=1}^M \ell(x_i(L+1)-y_i).
\end{equation}
Several methods have been proposed to solve this optimization problem,  e.g.~stochastic gradient descent~\cite{haber2018look} or ensemble Kalman filter~\cite{Kovachki2019,Watanabe1990,Alper}.

\subsection{Neural differential equations} \label{subsec:nde}

In~\eqref{eq:resnet} the layers define a discrete structure within the ResNet. This granularity has been interpreted as discretization in time of an underlying system of ODEs, known as {\em neural differential equations}.
To recover the time-continuous formulation of~\eqref{eq:resnet},  the following assumption is considered~\cite{chen2018neural}.
	
\begin{assumption} \label{ass:neurons}
	The number of neurons in each layer is given by $N_\kappa = N_0 = d$, $\forall\,\kappa=1,\dots,L+1$.
\end{assumption}
	
The choice $d=1$ is possible and, in addition, networks of this type have already been proved to satisfy different formulations of the universal approximation theorem~\cite{UniversalApproximator,LuLu2020,kidger2020universal}.
%
%
Under Assumption~\ref{ass:neurons} equation \eqref{eq:resnet} is an explicit Euler discretization of
\begin{equation} \label{eq:micro}
	\begin{cases}
	\dot{x}_i(t) = \sigma\big( w(t) x_i(t) + b(t) \big), \quad t\in[0,T]\\
		x_i(t_0) = x_i^0,
	\end{cases}
\end{equation}
for each $i=1,\dots,M$ and time-step $\Delta t$. The system of differential equations~\eqref{eq:micro} describes the time propagation of each measurement $x_i(t)\in\R^d$ and input data $x_i^0\in\R^d$.

In the continuous limit, the parameters of the network are time time-dependent functions, the weights $w(t)\in\R^{d\times d}$, and the bias $b(t)\in\R^d$. For fixed time $T$, then the training process~\eqref{eq:loss} turns into an optimal control problem: 
\begin{equation} \label{eq:lossMicro}
	\min_{(w(t),b(t))} \frac{1}{M} \sum_{i=1}^M \ell(x_i(T)-y_i) \ \mbox{ subject to } \eqref{eq:micro}.
\end{equation}

\subsection{Mean-field neural networks}\label{subsec:mf_nn}

The computational and memory cost of the training of a neural network, increases with the dimension $M$ of the data set. For this reason, an approach based on statistical mechanics leads to  the mean-field limit of~\eqref{eq:micro}  for $M\to \infty$, see~\cite{HTV_MFResNet}. The time-dependent probability measure $\mu_t:=\mu(t,\cdot)\in\mathcal{P}_1(\R^{d})$ fulfills the nonlinear transport equation
\begin{equation} \label{eq:meso}
		\partial_t \mu_t(x) + \nabla_x\cdot \Big(\sigma\big(w(t)x+b(t)\big)\mu_t(x)\Big) = 0, \quad t>0. 
\end{equation}
The initial condition $\mu_0 \in\mathcal{P}_1(\R^{d})$ is obtained as limit (in the sense of measures) of  $\lim_M \frac1M \sum_{i=1}^M \delta(x-x_{i,0})$. The well-posedness of equation~\eqref{eq:meso}, the continuous dependence on $(w,b)$ and the convergence of~\eqref{eq:micro} to~\eqref{eq:meso} in $1$-Wasserstein distance has been established in ~\cite{HTTV_training}  under the following assumptions:
\begin{align*}
	\text{(A1)} & \quad \quad \sigma \in C^{0,1}(\R^d), \ w, b \in C^{0,1}(\R), \\
	\text{(A2)} & \quad \quad |\sigma(x) | \leq C, \ \forall x \in \R^d, \ \text{ for some constant $C>0$}. 
\end{align*}


The training problem \eqref{eq:lossMicro} allows for a mean-field description as an optimal transport problem 
\begin{equation} \label{eq:lossMeso}
    \min_{(w(t),b(t))} \int_{\R^{d}} \tilde{\ell}(x) \mathrm{d}\mu_T(x) \ \mbox{ subject to }
     \eqref{eq:meso}.
\end{equation}
Here, $\tilde{\ell}$ is the mean-field limit of the loss function 
\begin{equation} \label{eq:lossMF}
    \tilde{\ell}(x) = \int_{\R^{d}} \ell(x-y) \de \nu(y),
\end{equation}
where $\nu\in\mathcal{P}_1(\R^d)$ is the probability measure of the target data for $M\to \infty$. The existence of solutions to \eqref{eq:lossMeso} has been established in \cite{HTTV_training}. Note that, the  optimal solution $(w,b)^*(\cdot;M)$ of problem \eqref{eq:lossMicro} depends on the number of particles $M.$ It still an open problem if the optimal solutions $(w,b)^*(\cdot;M)$ also converge to the optimal solution to  problem \eqref{eq:lossMeso} for $M\to \infty.$


\section{Controllability problems}\label{sec:control_pb}

The optimal controls obtained as the solution to 
\eqref{eq:lossMicro} or \eqref{eq:lossMeso}, respectively, depend on the choice of $\ell.$ In this section, we consider a different approach based on controllability. This aims to find a suitable mapping parameterized by $(w,b),$ such that it's possible to map the initial data on the target data. The corresponding definition~\cite{Coron,Zuazua} for linear time-variant systems is given below:



\begin{definition}[Microscopic controllability] \label{def:micro:controllab}
    Let $A(t)\in\mathbb{R}^{Md\times Md}$ and $B(t)\in\mathbb{R}^{Md\times d}$ be time-dependent matrices, and let $x\in C^0 ([t_0,T]; \mathbb{R}^{Md})$ satisfy
	\begin{equation}\label{eq:system_Coron}
		\dot{x}= A(t) x + B(t) u, \quad t\in [t_0,T].
	\end{equation}
	Then, the linear time-varying control system~\eqref{eq:system_Coron} is \textit{controllable} if, for every $(x^0,y)\in \mathbb{R}^{Md} \times \mathbb{R}^{Md}$, there exists a control $u \in L^\infty ((t_0,T); \mathbb{R}^d)$ such that the solution $x(t)$ of the Cauchy problem defined by~\eqref{eq:system_Coron} with initial condition $x(t_0)=x^0$ satisfies $x(T)=y$.
\end{definition}

The system~\eqref{eq:system_Coron} is a particular case of the neural ODEs in the  setting when the activation function $\sigma(x)=x$ is the identity as shown in the following in  Corollary~\ref{corr1}. The controllability in this particular setting relies on the Hilbert Uniqueness Method (HUM),  see e.g.~\cite{Coron}. We first state the general result before applying it to the system~\eqref{eq:micro}.

Let $\mathcal{R} \subset \R^d$ be the reachable set, i.e., the set of all states $y\in \R^{M d}$ such that there exists $u \in L^\infty ((t_0,T); \mathbb{R}^d)$ such that the solution $x$ to the Cauchy problem defined by~\eqref{eq:system_Coron} with initial condition $x(t_0) = 0$ satisfies $x(T)=y$.
For $\lambda^T$ given in $\mathbb{R}^{Md}$, we consider the solution $\lambda: [t_0,T]\rightarrow\mathbb{R}^{Md}$ of the following adjoint  problem:
\begin{equation}\label{eq:Cauchy_pb_back}
	\dot{\lambda}= -A(t)^{tr} \lambda, \quad \lambda(T) = \lambda^T, \ t \in [t_0,T].
\end{equation}
Then, the following holds true:
\begin{proposition}\label{prop:coron}
	Let $u \in L^2 ((t_0,T); \mathbb{R}^d)$. Let $x: [t_0,T]\rightarrow\mathbb{R}^{Md}$ be the solution of the Cauchy problem defined by~\eqref{eq:system_Coron} with initial condition $x(t_0) = 0$. Let $\lambda^T\in \mathbb{R}^{Md}$ and let $\lambda: [t_0,T]\rightarrow\mathbb{R}^{Md}$ be the solution of the Cauchy problem \eqref{eq:Cauchy_pb_back}. Then
\begin{equation*}
x(T)\cdot\lambda^T = \int_{t_0}^{T} u(t)\cdot B(t)^{tr} \lambda(t) dt.
\end{equation*}
\end{proposition}

The proposition allows to define a closed system in terms of $(x,\lambda)$: let $\Gamma$ be the map $\Gamma : \lambda^T\in\mathbb{R}^{Md} \mapsto x(T) \in \mathbb{R}^{Md}$, where $x$ is the solution of the Cauchy problem
\begin{equation*}
	\dot{x} = A(t)x + B(t)\bar{u}(t), \quad x(t_0) = 0, \quad \bar{u}(t):= B(t)^{tr} \lambda(t).
\end{equation*}
The latter now allows to state a controllability result, see \cite{Coron}. 
\begin{theorem}\label{thm:coron}
Under the assumptions of Proposition~\ref{prop:coron} the time-varying system \eqref{eq:system_Coron} is controllable, i.e., 
$$\mathcal{R}= \Gamma (\mathbb{R}^{Md}).$$
\end{theorem}

Due to the linearity of the system \eqref{eq:system_Coron},   Proposition \ref{prop:coron} and Theorem \ref{thm:coron} generalize to the general initial condition $x(t_0) = x^0$.


By applying Theorem~\ref{thm:coron} to the linear system of neural ODEs, we achieve the exact controllability of the system to an arbitrary final state, namely $x(T)= y$, using the bias $b(t)$ as control which can be explicitly computed.

\begin{corollary}\label{corr1}
Consider the neural ODE system \eqref{eq:micro} with $\sigma(x)=x$. Then, for any weights $w(t) \in L^\infty(0,T; \R^d)$  the system is controllable in the sense of Definition~\ref{def:micro:controllab} for every $(x_0,y) \in \R^{M d} \times \R^{M d}$. Further, the  control is given only by the bias $b \in L^\infty(0,T; \R^d).$ 
\end{corollary}
\begin{proof}
\smartqed
Since $\sigma(x)=x$, we note that $x(t)-x^0$ fulfills a time-variant linear system of the type \eqref{eq:micro} with initial condition $z(0)=0\in\R^{M d}.$
Then, equation \eqref{eq:micro} for $i=1,\dots,M$ is given by  system \eqref{eq:system_Coron} using the following definitions
\begin{align*}
\R^{Md\times Md}\ni& \quad A(t) =I_{M\times M} \otimes w(t),\\
\R^{Md\times d}\ni& \quad  B(t) = \mathbbm{1}_{M\times1} \otimes I_{d\times d}.
\end{align*}
Then, Theorem \ref{thm:coron} applies. Note that using the adjoint equation 
\begin{equation*}
	\dot{\lambda}= -A(t)^{tr} \lambda, \quad \lambda(T) = \lambda^T, 
\end{equation*}
which  fulfills 
$	x(T)\cdot\lambda^T = \int_{t_0}^{T} b(t) \cdot \sum_{i=1}^M \lambda_i(t) dt,
$
allows to characterize the optimal bias by
\begin{equation*}
	{b}(t)= \sum_{i=1}^M \lambda_i(t).
\end{equation*}
This concludes the proof.
\end{proof}

The results can be extended to the mean-field description of the controllability problem.
Here, we exploit the fact that the probability measure $\mu$ can be characterized by a push-forward mapping of the initial probability distribution. This will allow to apply 
controllability results from ODE theory to achieve mean-field controllability. The latter is defined in the following definition:

\begin{definition}[Mean-field controllability] \label{def:meso:controllab}
	Let $\mu_0,\nu\in\mathcal{P}_1(\R^d)$ be given. Let $T>0$ be fixed. We say that the mean-field equation~\eqref{eq:meso} is controllable if there exist $w \in C^{0,1}([0,T];\R^{d\times d})$ and $b \in C^{0,1}([0,T];\R^{d})$ such that $(\Phi_T\# \mu_0)=\nu$ where $\Phi_t:\R^d\to\R^d$ is the Lipschitz continuous characteristic flow of~\eqref{eq:meso}.
\end{definition}

As seen in the previous results, we expect controllability in the case of an activation  function that is the identiy. Next, we define the  flow in this particular case. Let  $\Phi:(t,x)\in\R^{1+d}\mapsto \Phi_t(x):=\Phi(t,x)\in\R^{d}$  satisfies
\begin{equation}\label{eq:flow}
	\begin{cases}
    	\partial_t \Phi(t,x) = w(t) \Phi(t,x) + b(t),  \\
		\Phi(t_0,x)  = x.
	\end{cases}
\end{equation}
The flow $\Phi_t$ is a weak solution to the neural mean-field equation~\eqref{eq:meso} with $\sigma(x)=x$, as
\[
\mu_t = \Phi_t  \symbol{35}  \mu_0,
\]
see~\cite{HTTV_training} for further details. Following, the discussion in the previous section, we consider the adjoint flow $\Lambda$ to 
\eqref{eq:flow} 
\begin{equation}\label{eq:flow_adj}
		\begin{cases}
			\partial_t \Lambda(t,x) = - \bigl(w(t)\bigr)^{tr} \Lambda(t,x),  \\
			\Lambda(T,x)  = \Lambda^T(x),
		\end{cases}
\end{equation}
where $\Lambda: [t_0,T] \times \R^{d} \rightarrow \R^d$.

Note that both the flow and the adjoint are now depending on the point $x \in \R^d$ and are therefore parameterized by $x.$ Also, the reachable set $\mathcal{R}_\Phi$ is now parameterized by $x$. Provided we have integrability with respect to the parameter $x$, we establish the following realation between flow and its adjoint, similar to Proposition \ref{prop:coron}. 
\begin{proposition}\label{prop:flow}
	Let $b \in L^2 ((t_0,T); \mathbb{R}^d)$ and $\Phi: [t_0,T]\times \R^d\rightarrow\mathbb{R}^{d}$ be the solution of \eqref{eq:flow}. Let $\Lambda^T(x)\in \mathbb{R}^{d}$ and $\Lambda: [t_0,T] \times \R^{d} \rightarrow \R^d$ be the solution of \eqref{eq:flow_adj}. Assume $\Phi(t,\cdot), \Lambda(t,\cdot) \in L^2(\R^d)$ a.e. in $t,$ then
	\begin{equation}\label{eq:coron-flow}
		\langle \, \Phi(T,x), \, \Lambda^T(x) \, \rangle_{L^2} = \int_{t_0}^{T}  \langle \, b(t), \, \Lambda(t,x) \, \rangle_{L^2}  \, dt.
	\end{equation}
\end{proposition}
\begin{proof} The relation is obtained by direct integration
    \begin{equation*}
    \begin{split}
    \langle \, \Phi(T,x), \, \Lambda^T(x) \, \rangle_{L^2}
    =& \int_{t_0}^{T}  \frac{d}{dt} \langle \, \Phi(t,x), \, \Lambda(t,x) \, \rangle_{L^2}  \, dt  \\
    =& \int_{t_0}^{T}  \langle \, w(t) \Phi(t,x) + b(t), \, \Lambda(t,x) \, \rangle_{L^2} \, \\
    &- \, \langle \, \Phi(t,x), \, \bigl(w(t)\bigr)^{tr} \Lambda(t,x) \, \rangle_{L^2}   \, dt     \\
    =& \int_{t_0}^{T}  \langle \, b(t), \, \Lambda(t,x) \, \rangle_{L^2}  \, dt.
    \end{split}
	\end{equation*}
 ~
\end{proof}

Using the adjoint we define  $\Gamma_\Phi$ as a the following map, i.e.
\begin{equation*}
	\begin{split}
		\Gamma_\Phi : L^2(\R^{d}; \R^{d})  &\rightarrow C^{0,1}(\R^{d}; \R^{d})  ,  \\
		\Lambda^T(x) &\mapsto \Phi(T, x),
	\end{split}
\end{equation*}
where $\Phi$ is the solution of \eqref{eq:flow}
and
\begin{equation}\label{eq:control_flow}
	b(t)=\bar{b}(t,x):= \Lambda(t,x).
\end{equation}

The following Theorem \ref{thm:flow} then extends Theorem \ref{thm:coron} and allows to to characterize $\mathcal{R}_\Phi$. The proof follows by realizing that the system is only a parameterized (by $x$) version of Theorem \ref{thm:coron}.
\begin{theorem}\label{thm:flow}
Under the assumptions of Proposition \ref{prop:flow}, we have	$$\mathcal{R}_\Phi= \Gamma_\Phi (\mathbb{R}^{d}).$$
\end{theorem}

Hence, the HUM method applies to the problem of controllability of equation \eqref{eq:flow} for fixed $x.$ This leads to a control  $b$ that may depend on  $x$ (and also eventually on the desired state $y$). In other words, the application 
of the previous Corollarly \eqref{corr1} will lead us to  $b = b(t,x,y)$. Therefore, in order to obtain a suitable mean-field corresponding to the flow
\begin{equation}\label{eq:flow:new}
	\begin{cases}
    	\partial_t \Phi(t,x) = w(t) \Phi(t,x) + b(t,x,y),  \\
		\Phi(t_0,x)  = x
	\end{cases}
\end{equation}
we reformulate it as an autonomous system with respect to $x,y$. To this end, we introduce auxiliary variables $x_1(t):=\Phi(t,x), \, x_2(t), x_3(t)  \in \R^d.$ This leads to the \emph{extended} flow
\begin{equation}\label{eq:flow_aut}
	\begin{cases}
		 \dot{x}_1(t) = w(t) x_1(t) + b(t,x_2(t),x_3(t)),  & x_1(0) = x,\\
		 \dot{x}_2(t) = 0, & x_2(0) = x,\\
		 \dot{x}_3(t) = 0,  & x_3(0) = y.\\
	\end{cases}
\end{equation}

To simplify the presentation, let us denote by $\xi(t,x_1,x_2,x_3) := (x_1,x_2,x_3) $ with  $x_1(t):=\Phi(t,x)$ and the right hand side of \eqref{eq:flow_aut} as

\begin{equation*}
	\begin{aligned}
		G_t: \mathbb{R}^{3d} &\to \mathbb{R}^{3d}, \; 
		(x_1,x_2,x_3) &\mapsto G_t	(x_1,x_2,x_3) := \partial_t \xi(t,x_1,x_2,x_3),
			\end{aligned}
\end{equation*}
and
\begin{equation}\label{eq:flow_xi}
	\begin{cases}
		\partial_t \xi(t,x_1,x_2,x_3) = \Big( w(t) x_1 + b(t,x_2,x_3), \, 0, \, 0 \Big)^{tr},\\
		\xi(t_0,x_1,x_2,x_3) =  \Big( x, \, x, \, y \Big)^{tr}.
	\end{cases}
\end{equation}

The push-forward of a measure $\bar{\mu}_0$ with the flow $\xi_t$ defines  a  measure  $\bar \mu_t$, belonging to the set $\mathcal{P}_1(\R^{3d})$, and satisfying in the weak sense, the mean-field equation
\begin{equation}\label{eq:bar_mu}
\partial_t \bar{\mu}_t(x_1,x_2,x_3) + \nabla_{x_1}\cdot \Big( \left(w(t)x_1+b(t,x_2,x_3) \right) \bar{\mu}_t(x_1,x_2,x_3) \Big) = 0.
\end{equation}

The following proposition, shows that the empirical measure \eqref{emp}  is  a  solution of our new mean-field PDE \eqref{eq:bar_mu} in a weak sense. 

\begin{proposition}\label{ex:dirac}
   The empirical measure $\bar{\mu}^M$ defined by 
\begin{equation}\label{emp} 
\bar{\mu}^M(t,x_1,x_2,x_3)=\frac1M\sum_{i=1}^M\delta(x_1-x_i(t)| x_i^0,y_i) \delta(x_2 - x_i^0)  \delta(x_3 - y_i),\end{equation}
    where 
    \begin{equation} \label{eq:micro_linear}
	\begin{cases}
		\dot{x}_i(t) = w(t) x_i(t) + b(t,x_i^0,y_i), \\
		x_i(t_0) = x_i^0, \quad i = 1,\dots,M
	\end{cases}
\end{equation}
is a  solution of the mean-field equation \eqref{eq:bar_mu} in the distributional sense. 
\end{proposition}

\begin{proof}
    Multiplying Eq. \eqref{eq:bar_mu} by a smooth compactly supported test function $\psi \in C^\infty_C(\R^{d})$, integrating with respect to $x_1,x_2,x_3$, and substituting $\bar \mu_t = \bar \mu_t^M$, we obtain

	\begin{equation*}
	\begin{split}
	0
		&=  \partial_t \int_{\R^{3d}} \psi(x_1) \,  \bar{\mu}^M(t,x_1,x_2,x_3) \, d{x_1} d{x_2} d{x_3} + 
		\\
		& \quad - \int_{\R^{3d}} \nabla_{x_1}\psi(x_1) \,   \bigl(  (w(t)x_1+b(t,x_2,x_3) \bigr) \,\bar{\mu}^M(t,x_1,x_2,x_3) \, d{x_1} d{x_2} d{x_3} \\
		&=  \frac1M\sum_{i=1}^M \nabla_{x_{i}(t)} \psi(x_{i}(t)) \Bigl( \dot{x}_i(t) - w(t)x_i(t)-b(t,x_i^0,y_i)\Bigr).
	\end{split}
\end{equation*}
This shows that $\bar{\mu}^M$ is a solution in distributional sense, provided that  $\dot{x}_i(t) - w(t)x_i(t)-b(t,x_i^0,y_i) = 0$. 
\end{proof}

We furthermore have the following result using the flow. 
\begin{proposition}
Let $\bar{\mu}_0 \in \mathcal{P}_1(\R^{3d})$ be a probability measure and $\xi_t$ the flow satisfying \eqref{eq:flow_xi}, then the pushed-forward measure $\bar{\mu}_t \in \mathcal{P}_1(\R^{3d})$ defined as
$$ \bar{\mu}_t = \xi_t \symbol{35} \bar{\mu}_0$$
fulfills the Cauchy problem for the mean-field equation \eqref{eq:bar_mu} with initial condition $\bar{\mu}_0$ in distributional sense. 
\end{proposition}
\begin{proof}
	Let $\psi \in C^\infty_C(\R^{3d})$ be a smooth compactly supported test function, define the variable $z:= (x_1,x_2,x_3) \in \R^{3d}$, and then 
	\begin{align*}
				\int_{\R^{3d}} \psi(z) \, \partial_t \bar \mu(t,z)\, dz
			= \frac{d}{dt} \int_{\R^{3d}} \psi(z) \, (\xi_t \symbol{35}  \bar{\mu}_0)(z) \, dz = \\ \frac{d}{dt} \int_{\R^{3d}} \psi \bigl(\xi_t(z^0)\bigr) \, \bar{\mu}_0(z^0) \, dz^0 
				= 	 \int_{\R^{3d}} \nabla_z \psi \bigl(\xi_t(z^0)\bigr)\,  \partial_t  \xi_t(z^0) \, \bar \mu_0(z^0) \, dz^0 \\ 
				= 	 \int_{\R^{3d}}  \nabla_z \psi (z)\,  G_t(z)\, (\xi_t \symbol{35}  \bar \mu_0)(z) \, dz 
				= 	 \int_{\R^{3d}}  \nabla_z \psi (z)\,  G_t(z)\, \bar \mu (t,z)\, dz.\\
			\end{align*}
   This concludes the proof.
   \end{proof}

The results presented in~\cite{HTTV_training} regarding well-posedness and convergence of the empirical measure $\bar{\mu}^M$ towards $\bar{\mu}$ in $1-$Wasserstein for $M\to \infty$ hold true.  Also,  Proposition~\ref{prop:flow} and Theorem~\ref{thm:flow} concerning the controllability extend to the flow defined by $\xi_t$.

\subsection{Explicit example of Controllability}
We close the section by stating an explicit example of the computation of a static control. In this case, the associated mean-field characteristic flow evolves as follows: 
\begin{equation}\label{eq:flow_time_indep}
	\begin{cases}
		\dot{\Phi}(t)= w(t) \Phi(t) + b(x^0, y), \\
		\Phi(t_0) = x^0,
	\end{cases}
\end{equation}
where, for the sake of simplicity in our notation, we define $\Phi(t):=\Phi(t,x^0)$. 
Our objective is to steer the system from $x^0$ to the desired target $y$ using a control strategy, even if it may not be optimal, that remains constant in time.

Solving the inhomogeneous linear Cauchy problem \eqref{eq:flow_time_indep}, we obtain the solution
$$ \Phi(t) = c_1(t) \Bigl(  x^0 +  c_2(t) \, b(x^0, y) \Bigr)$$
where
$$c_1(t) = \exp\Bigl( \int_{t_0}^{t} w(s) \, ds \Bigr), \quad c_2(t) = \int_{t_0}^{t} \exp\Bigl(- \int_{t_0}^{s} w(r) dr \Bigr) \, ds.$$ 
By imposing the terminal condition $\Phi(T) = y$ and solving for the control $b$ we can derive
\begin{equation}\label{eq:static_control}
b(x^0,y) = \frac{y-x^0 c_1(T)}{c_1(T)c_2(T)}.
\end{equation}

To provide the reader with an understanding of how the static control behaves, we explore the following distinct scenarios, focusing on different configuration of the weights $w(t)$

\begin{description}
	\item[\textbf{Case (1).}] If $w(t) = \omega t^{\alpha}:$
		$$\Phi(t) = x^0 \, e^{\frac{\omega}{\alpha +1}(t^{\alpha +1}-t_0^{\alpha+1})}  
		+ b(x^0,y) \, e^{\frac{\omega}{\alpha +1}t^{\alpha +1}}   \int_{t_0}^{t}  e^{-\frac{\omega}{\alpha +1}s^{\alpha +1}}  ds. $$
	\item[\textbf{Case (2).}] If $w(t) = \omega e^{\alpha t}:$
		$$
	 \Phi(t) = x^0  \, e^{\frac{\omega}{\alpha}(e^{\alpha t}-e^{\alpha t_0})}  
	+ b(x^0,y)  \, e^{\frac{\omega}{\alpha}e^{\alpha t}}    \int_{t_0}^{t}  e^{-\frac{\omega}{\alpha}e^{\alpha s}}  ds. $$
\end{description}
\begin{figure}[t]\label{fig:decay}
	\begin{center}
\includegraphics[width=0.5\linewidth]{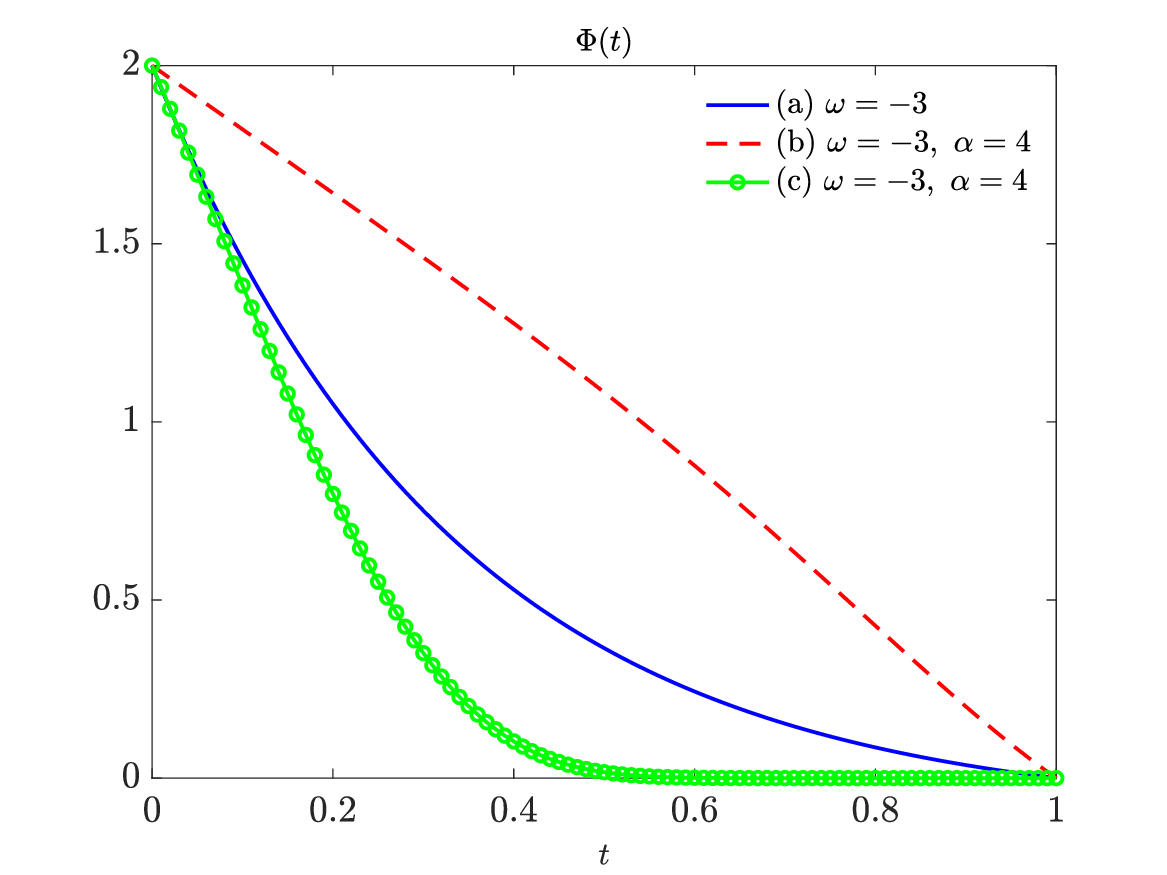}
\caption{Time evolution of $\Phi(t)$ towards the target value $y = 0$ for different weights configurations. The blue solid line (a) and the red dashed line (b) represent \textbf{Case (1)} with $\alpha = 0$ and $\alpha = 4$ respectively. The green circle line (c) represent \textbf{Case (2)} with $\alpha = 4$.}
	\end{center}
\end{figure}
In Figure \ref{fig:decay}, we illustrate the decay towards the target state $y = 0$. The initial condition is $x^0 = 2$, and the time interval $[0,1]$ is discretized with a time step $\Delta t = 0.01$. In all cases, despite the control being static and not optimal, we consistently achieve the desired decay
$ | \Phi(t)-y|\rightarrow0$, as $t\rightarrow T$. In particular, with $\omega = -3$ and $\alpha = 4$ we can observe the fastest decay for the \textbf{Case (2)}, that is an exponential type of weight $w(t)$. 
\section{A numerical approach for general training based on the kernel-based learning method}\label{subsec:kernelmethod}

The previous results show that the continuous ResNets is controllable if the activation function is linear, see also~\cite{Zuazua:neuralODE,Osher:controllability} for additional results. This provides in simple cases explicit solutions. However, to also allow for  training in the case of  nonlinear activation functions and to incorporate the weights, we propose an efficient training method. The computational framework to determine optimal parameters $w$ and $b$ in continuous ResNets will be based on a kernel method, to approximate the loss function, in both the microscopic case~\eqref{eq:lossMicro} and the mean-field limit~\eqref{eq:lossMeso}.

Kernel methods are supported by a well-developed theory~\cite{steinwart2008support,shawe2004kernel}, and come with efficient algorithms~\cite{scholkopf2002learning,williams2006gaussian,kanagawa2018gaussian}. From a mathematical point of view, these methods rest on the concept of kernels and function spaces generated by kernels, so-called reproducing kernel Hilbert spaces. In recent years, there has been an increasing interest in the application of kernel methods approaches in the context of interacting particle systems, see e.g.~\cite{CFiedler2023KRM}.

Define {the function $\mathcal{L}$} on the space $X:=\R^{d\times d}\times \R^{d}$ of possible pairs of control $(w,b)$ by
\begin{align*}
\mathcal{L}: \ X \ \rightarrow \ \R^+_0 \,, \quad
                              (w,b) \ \mapsto  \    \mathcal{L}((w,b)) \, .
\end{align*}

Since the loss function of a typical neural network contains the summation of possibly many target values, different techniques are known to efficiently compute its gradient, e.g. the stochastic gradient method. Here, we propose to approximate $\mathcal{L}$ by a function  $\widehat{\Lcal}$ that is fast and simple to evaluate with respect to the parameters $(w,b).$ Provided $\widehat{\Lcal}$ as an accurate approximation $\Lcal$,  we expect that training of $\Lcal$ can be replaced by a more efficient training of $\widehat{\Lcal}.$ The approximation should be consistent on the microscopic and mean-field level.   
Therefore, we consider kernel-based methods that have recently been shown to allow for an extension in the case of infinitely many agents (or data points $M$ in our case)~\cite{CFiedler2023KRM}.

The approximation  $\widehat{\Lcal}$  is given by a weighted sum 
\begin{equation}\label{eq:cal_L}
	\widehat{\Lcal}(w,b)  = \sum_{n=1}^N \alpha_n \, k\left(\, (w,b) \, , \, \bigl(w,b\bigr)_n\right),
\end{equation}
where $\alpha_1,\ldots,\alpha_N\in\R$ are coefficients, $\bigl(w,b\bigr)_1, \ldots,\bigl(w,b\bigr)_N\in X$ are given values and $k: X\times X\rightarrow\R$ is a kernel function, cf.\ e.g.\ \cite{scholkopf2001generalized}. To specify the approximation we need to determine suitable choices for $(w,b)_n$ and then compute the corresponding coefficients $\alpha_n$ of the approximation. Both will be detailed below.

The set of all  functionals $\widehat{\Lcal}$ represented by the series \eqref{eq:cal_L} forms a Hilbert space. Hence, the question is closely related to the problem of describing the Hilbert space generated by the kernel $k$, the so-called reproducing kernel Hilbert space.
For the convenience of the reader, some definitions and concepts on reproducing kernel Hilbert spaces, see \cite[Chapter~4]{steinwart2008support}.
	\begin{definition} Let $\mathcal{X} \not = \emptyset$ be an arbitrary set and $H \subseteq \R^\mathcal{X}$ a real Hilbert space of functions. Then, the function $k: \mathcal{X} \times \mathcal{X} \rightarrow \R$ is called a kernel, if  there exists a real Hilbert space $\mathcal{H}$ and a map $\Phi: \mathcal{X} \rightarrow \mathcal{H}$ 
		\begin{equation*}
			k(x,x') = \langle \Phi(x'), \Phi(x) \rangle_\mathcal{H} \quad \forall x,x'\in \mathcal{X}.
		\end{equation*}
		The map $\Phi$ is called the feature map and $\mathcal{H}$ is the  \defm{feature space}.  
				\end{definition} 
			\begin{definition}
		The mapping $k$ is called a \defm{reproducing kernel} (for $H$) if 
		$k(\cdot,x)\in H \quad \forall x\in \mathcal{X}$ and if $f(x)=\langle f, k(\cdot, x)\rangle_H \quad \forall f\in H, x \in \mathcal{X}.$
		If a kernel $k$ has the previous property, then $k$ is said to have  the \defm{reproducing property}.
			\end{definition} 
	\begin{definition}
		$H$ is called a \defm{reproducing kernel Hilbert space} (RKHS) if all evaluation functionals are continuous, i.e., if 
		for all $x \in \mathcal{X}$ the functionals $\delta_x: H \rightarrow \R, \; \delta_x(f)=f(x)$ are bounded.
			\end{definition} 
		Finally, we say that $k$ is positive definite  if for all $N\in\mathbb{N}$, $x_1,\ldots,x_N\in\mathcal{X}$ and $\alpha_1,\ldots,\alpha_N\in\mathbb{R}$ we have
		\begin{equation*}
			\sum_{i,j=1}^N \alpha_i \alpha_j k(x_j,x_i) \geq 0.
		\end{equation*}
Note that only symmetric and positive definite functions $k$ can be kernels. We, therefore, use   Gaussian kernels with the particular choice 
 \[k_\gamma(x,x')=\exp\left(-\|x-x'\|^2/2\gamma \right),\]
  where $\gamma\in\Rp$ plays the role of a length scale.
For more details on this kernel and its associated RKHS, we refer to \cite[Section~4.4]{steinwart2008support}.

The idea of the approximation problem \eqref{eq:cal_L} is as follows. Assume we have $N$ possibly noisy measurements of the feature map 
\begin{align}\label{inference} 
z_i = \Lcal( (w,b)_i) + \eta, \quad i=1,\dots,N,
\end{align}
for noise $\eta$. The latter can be  distributed e.g. $\mathcal{N}(0,\Sigma).$ The regression problem is solved e.g. by support vector machines \cite{steinwart2008support}  or Gaussian process estimation \cite{williams2006gaussian}, being two particular kernel methods.
Using kernel methods, problem \eqref{inference} can be solved by $\widehat{\Lcal}$ belonging to the RKHS $H$ given as a solution to a minimization problem for a continuous, strictly convex function ${\bf L}$ and $\lambda>0:$
\begin{align}\label{inference2}
	\widehat{\Lcal} = \mbox{argmin}_{\Lcal \in H} \; {\bf L}\left( \Lcal((w,b)_1), \Lcal((w,b)_2), \dots, \Lcal((w,b)_N) \right) + \lambda \| \Lcal \|^2_{H}\,,
\end{align}
with
\begin{align}\label{regression2}
{\bf L}=\frac{1}N \sum\limits_{i=1}^N \| z_i - \Lcal( (w,b)_i) \|^2_\Sigma. 
\end{align}
 The norm in \eqref{inference2} is given by $\| \Lcal \|_{H}^2 = \sum\limits_{i=1}^N \sum\limits_{j=1}^N \alpha_i \alpha_j k\left(\, (w,b)_i \, , \, \bigl(w,b\bigr)_j\right)$ and the parameter $\lambda>0$ should be chosen according to a discrepancy principle. 

In the numerical tests of sections \ref{sec:kernel:micro} and \ref{sec:kernel:meso}, we consider the noisy free case (i.e. 
$\Sigma = 0$ and $\lambda = 0$), hence the unknown coefficients $\alpha_n$ in~\eqref{eq:cal_L}, are obtained by interpolation. This leads to the linear system:
\[
\sum_{n=1}^N \alpha_n k \bigl(\,(w,b)_m , \, (w,b)_n \, \bigr) = \mathcal{L}((w,b)_m), \qquad m=1,\ldots,N.
\]

The training algorithm we propose consists of determining $\widehat{\Lcal}$ and considering the optimization of  the function $\widehat{\Lcal}$ with respect to the variables $(w,b)$. Numerically, we apply  a gradient descent method and also treat the case of nonlinear activation functions. Since both, the optimization problem, as well as the kernel approximation, allow for a mean--field description. Hence, we expect performance results independent of the number of data points $M$.


\subsection{Microscopic control problem} \label{sec:kernel:micro}

In this Section we present  numerical tests concerning the microscopic control problem introduced in~\eqref{eq:lossMicro}.  In particular, we are interested in finding the optimal values of both $w$ and $b$ constant in time using the kernel method.

The learning tasks involve {the function $\mathcal{L}$} on the space of possible pairs $(w,b)$, where
\begin{equation*}
\mathcal{L}\bigl( \, (w,b) \, \bigr) = \frac{1}{M}\sum_{i=1}^{M} \ell ( x_i(T) - y_i),
\end{equation*}
and where $x_i \in \R^d, \ i=1,\dots,M$ evolves in time according to \eqref{eq:micro} with the specific choice of weights $w\in\R^{d\times d}$ and bias $b\in\R^d$. In the numerical experiments, we consider $d=1$ and $\ell(z) = | z |^2$.

With the kernel method, we generate the approximation of $\Lcal$ as in~\eqref{eq:cal_L} for a fixed  $N=20$ and given observations $(w,b)_n$, $n=1,\dots,N$.
In the following numerical example, we consider the Gaussian kernel parameter $\gamma = 10^{-2}$ and $\| \cdot \|$ the usual Euclidean norm.

In the microscopic non-linear setting, we take $M = 50$ as size of the data-set, a ReLU activation function $\sigma(z) = \max\{0,z\}$, and a time-step $\Delta t = 10^{-2}$ for the discretization of the time interval $[t_0,T]$ with $t_0=0$ and $T=10$. The initial input data is uniformly distributed, namely $x_i^0 \sim \mathcal{U}([1,2])$, for every $i=1,\dots,M$.

The parameters space is chosen as $(w, b) \in [\, 0, \, 2.5\cdot 10^{-1}\, ] \, \times \, [\, 0, \, 2.5\cdot 10^{-3}\, ]$ which is discretized with a uniform grid of size $10^{-2}$. This leads to a total number of $N_{tot}=676$ pairs $(w, b)$. We recall that we have access to the observed value of $\Lcal$ just for $N=20$ out of the  $N_{tot}=676$ weights-bias pairs.

We choose the target $y_i=y$ for every $i=1,\dots,M$ computed by
$$y = \frac{1}{M}\sum_{i=1}^{M} \tilde{x}_i(T),$$
where $\tilde{x}_i(T)$ represents the state at final time $T$ of the input signal $x_i^0$ evolved according to \eqref{eq:micro} with the value of $(w,b)$ chosen in the centre of the parameter grid. This artificially corresponds to a classification problem with a single classifier. 

\begin{figure}[t]
\begin{center}
\includegraphics[width=0.24\linewidth]{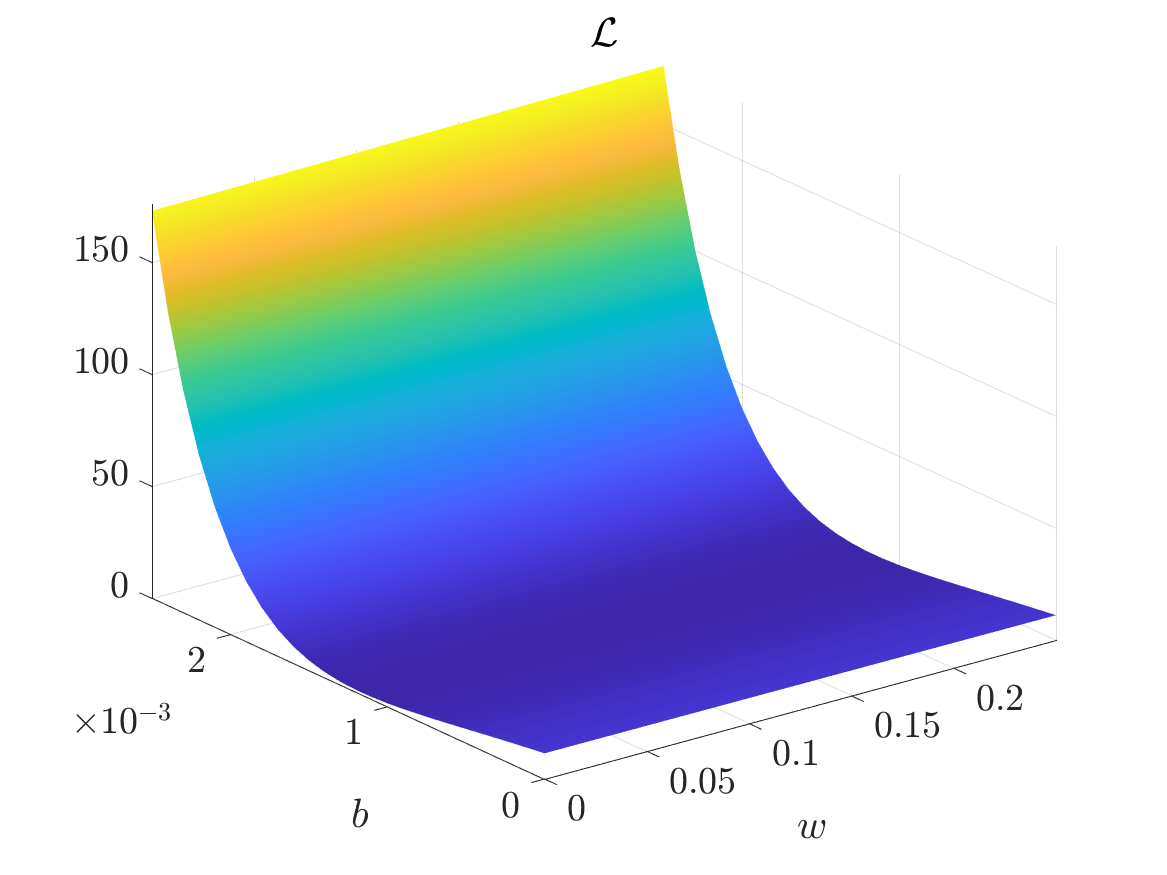}
\includegraphics[width=0.24\linewidth]{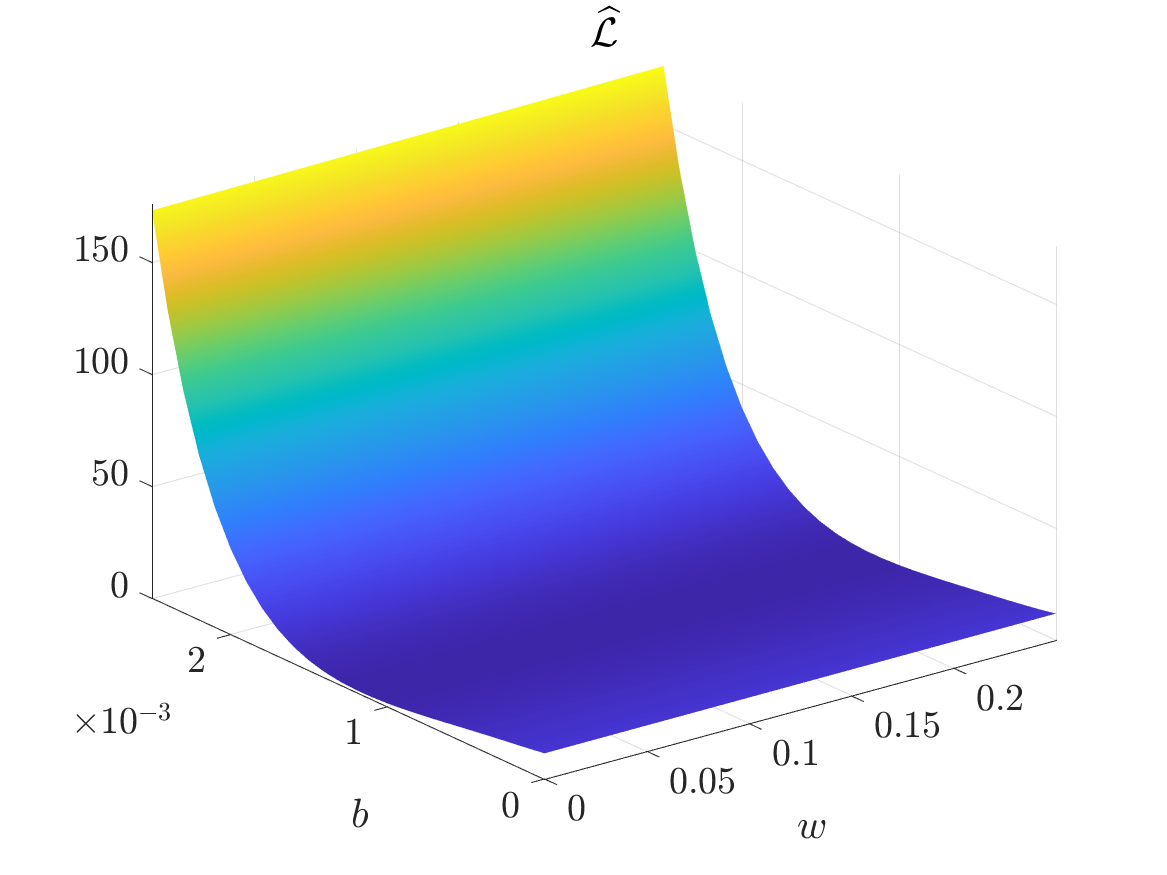}
\includegraphics[width=0.24\linewidth]{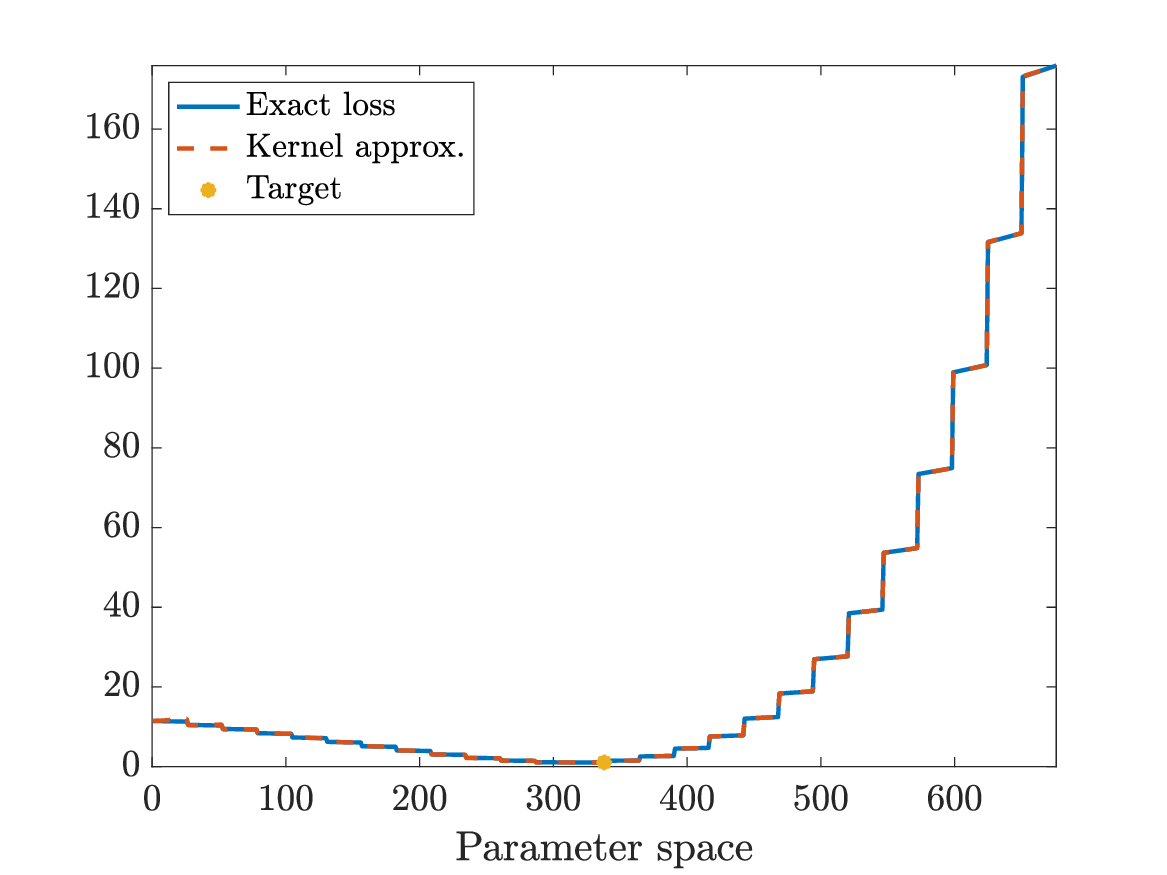}
\includegraphics[width=0.24\linewidth]{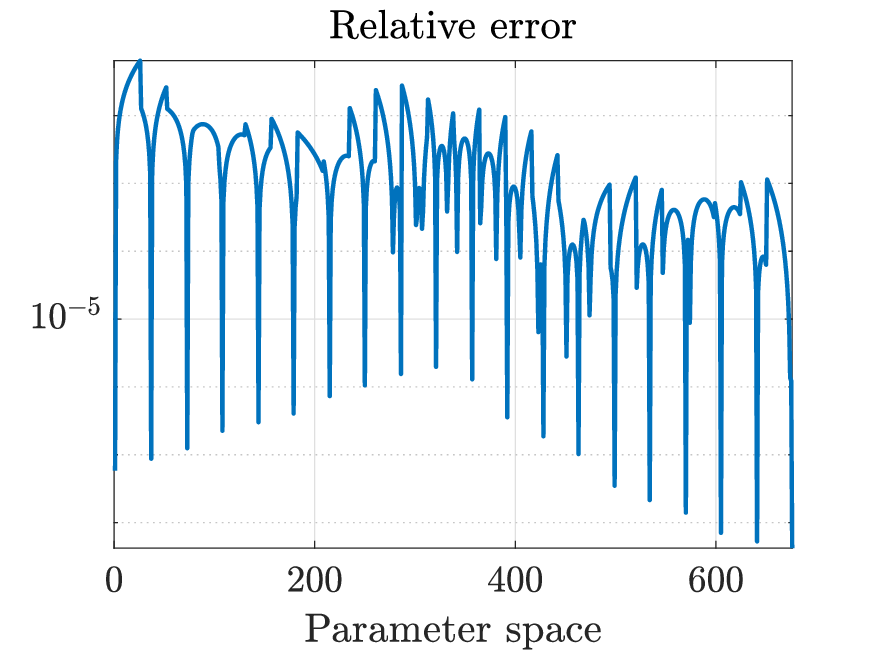}
\caption{Microscopic case. Leftmost panel: exact loss functional computed on the parameter grid. Center-left panel: approximated loss functional via kernel method with $N=20$ points. Center-right panel: comparison between $\mathcal{L}$ and $\widehat{\mathcal{L}}$. Rightmost panel: relative error between $\mathcal{L}$ and $\widehat{\mathcal{L}}$.}
\label{fig:1delta}
\end{center}
\end{figure}

In Figure~\ref{fig:1delta} we show a comparison between the exact loss functional $\mathcal{L}$ and the approximated functional $\widehat{\mathcal{L}}$. We observe that there is a good agreement, in fact, the relative error $(w,b)\mapsto |\mathcal{L}-\widehat{\mathcal{L}}|/|\mathcal{L}|$ is in the interval $[4\cdot10^{-9},5\cdot10^{-2}]$.

The approximation $\widehat{\Lcal}$ of the exact loss allows us to reformulate the control problem as a simpler minimization of the function $\widehat{\Lcal}$ with respect to $(w,b)$. We numerically look for the minimum using a projected gradient descent method in order to satisfy the box constraint given by the parameter domain. Specifically, the gradient method is solved for simplicity
with a constant stepsize of value $10^{-5}$. In Figure~\ref{fig:gd} we observe that $(w,b)$ converges towards the optimal parameters, and the cost decreases approaching the optimal. Moreover, we compare the trajectory of the mean of the data, noticing that the target is well approximated.

\begin{figure}[t]
\begin{center}
\includegraphics[width=0.32\linewidth]{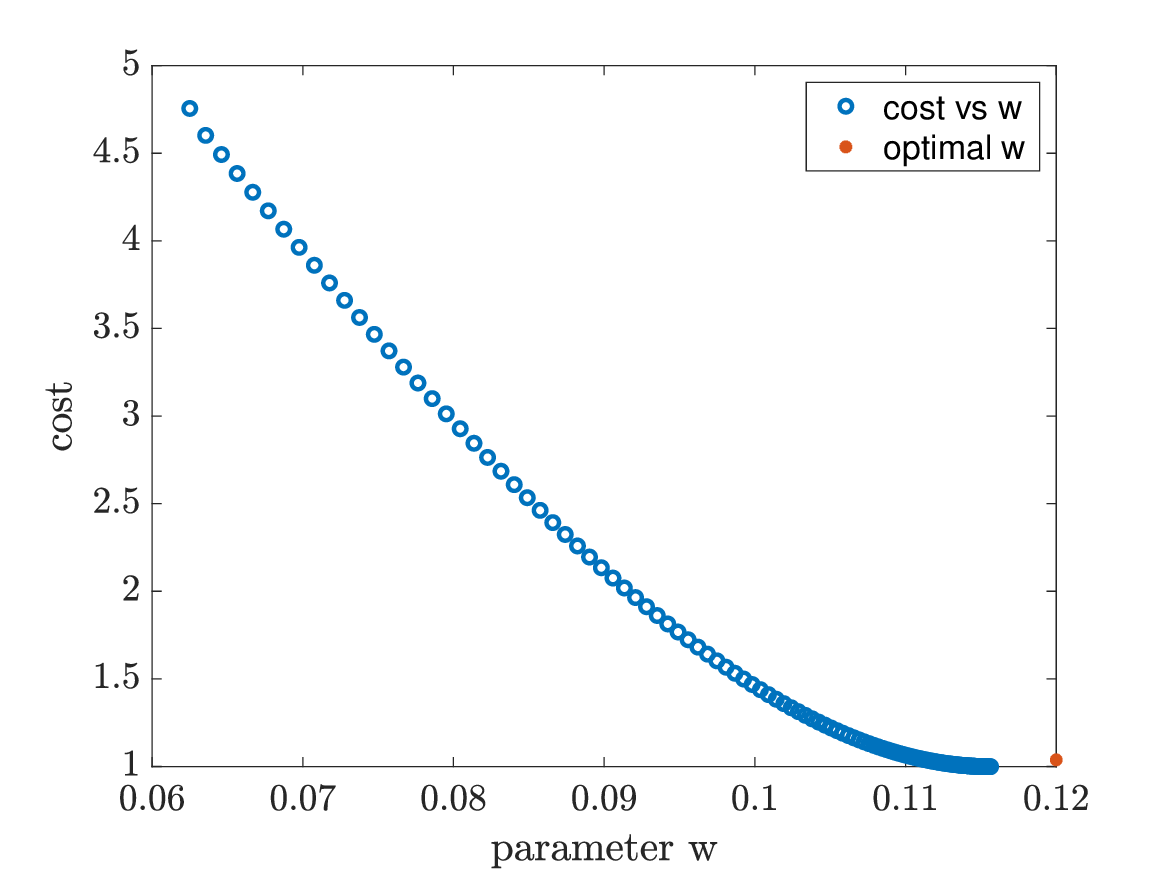}
\includegraphics[width=0.32\linewidth]{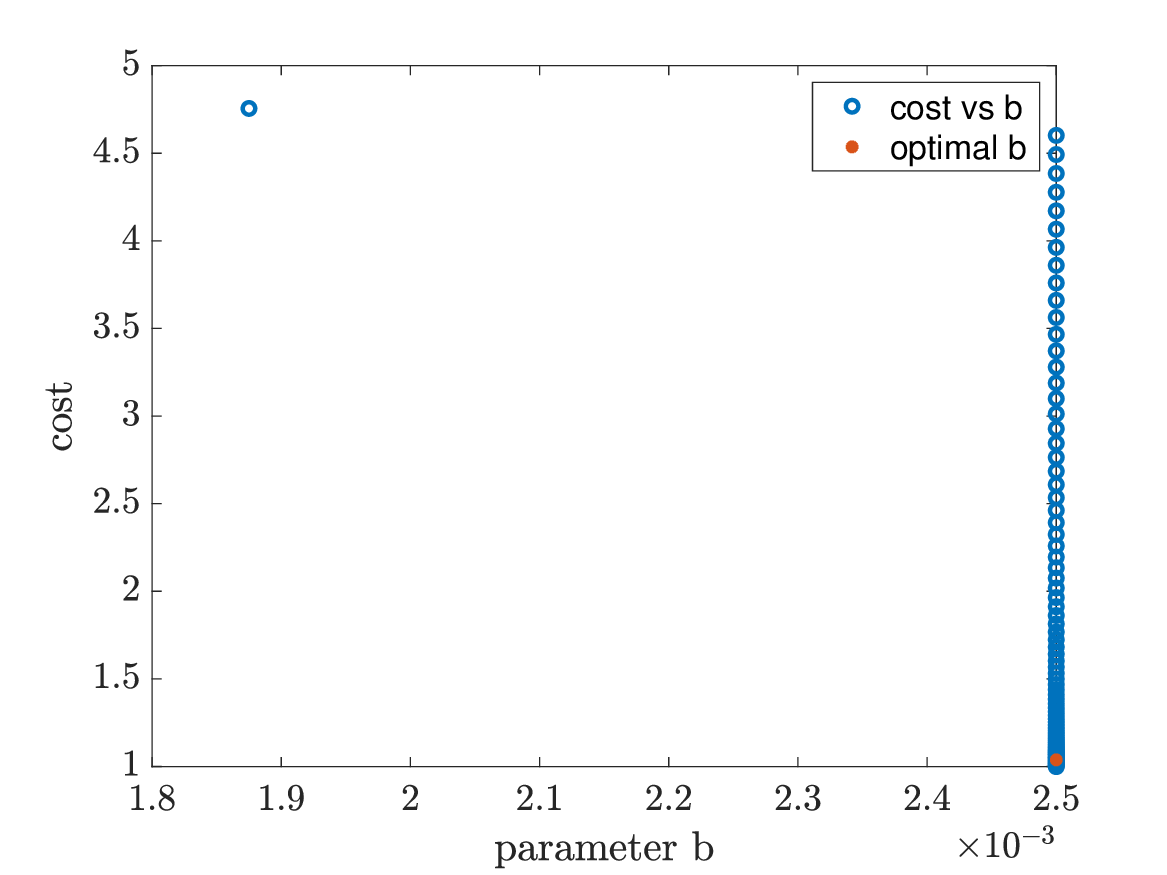}
\includegraphics[width=0.32\linewidth]{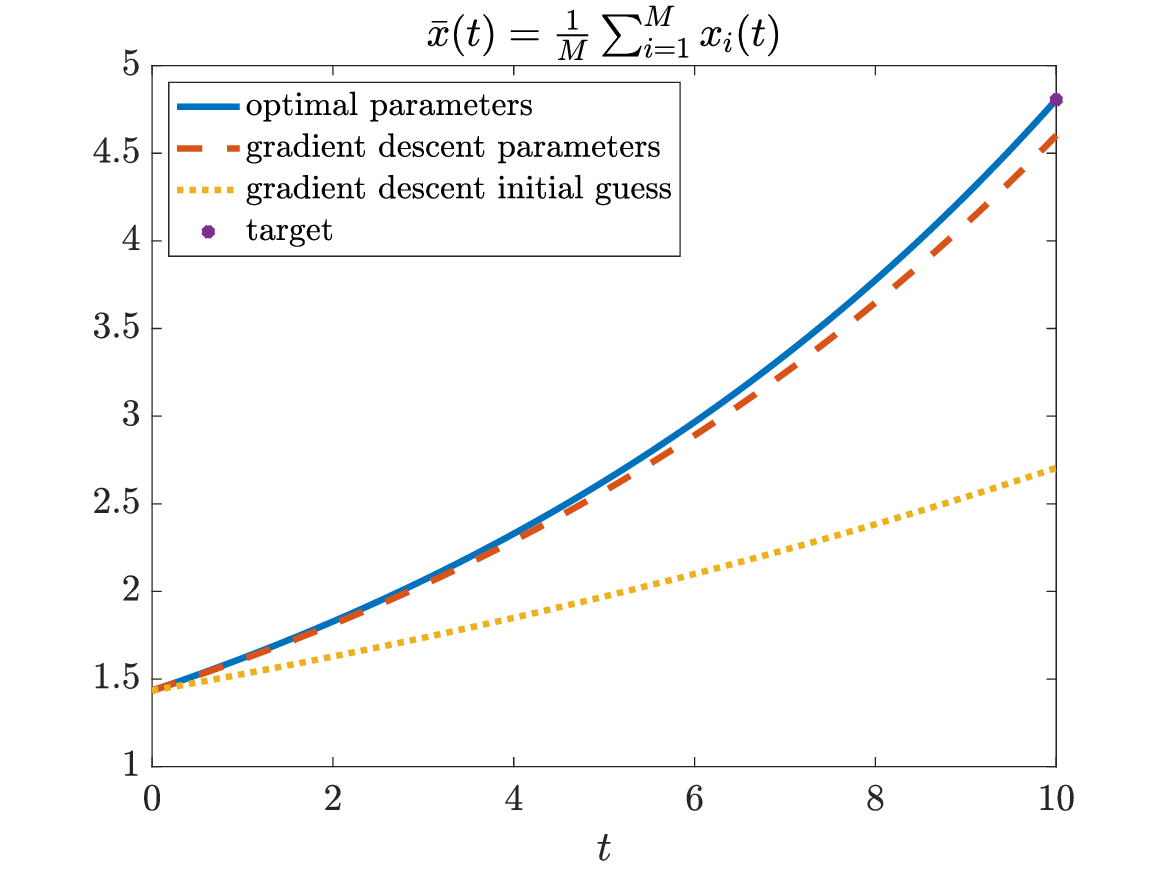}
\caption{Microscopic case. Left: $w$ vs.~the cost during the iteration of the gradient method. Center: $b$ vs.~the cost during the iteration of the gradient method. Right: evolution of the mean of the data with different parameters.}
\label{fig:gd}
\end{center}
\end{figure}

The approach can be easily generalized to the case of a classification with more classifiers. The results are  similar and therefore skipped in the presentation.

\subsection{Mean-field control problem} \label{sec:kernel:meso}

The learning tasks involve {the function $\mathcal{L}$} on the space of possible (constant) pairs $(w,b)$, where
\begin{equation} \label{eq:mesoL}
\mathcal{L}\bigl( \, (w,b) \, \bigr) = \int_{\R^{d}}  \int_{\R^{d}} \ell(x-y) \de \nu(y) \mathrm{d}\mu(T,x) 
\end{equation}
where $\nu\in\mathcal{P}_1(\R^d)$ is the probability measure of the target data and $\mu_T\in\mathcal{P}_1(\R^d)$ is the solution of the mean-field equation at time $T$ with the specific choice of weights $w\in\R^{d\times d}$ and bias $b\in\R^d$. In the numerical tests, we consider $d=1$ and $\ell(z) = | z |$.

We use the kernel method to approximate $\Lcal$ with $N=20$ and given observations $(w,b)_n$, $n=1,\dots,N$. The coefficients $\alpha_n$, $n=1,\dots,N$, are computed again by interpolation as discussed in Section~\ref{sec:kernel:micro}. The parameters space is chosen as $(w, b) \in [\, 0, \, 2.5\cdot 10^{-1}\, ] \, \times \, [\, 0, \, 2.5\cdot 10^{-3}\, ]$ which is discretized with a uniform grid of size $2\cdot10^{-2}$. This leads to a total number of $N_{tot}=169$ pairs $(w, b)$. To approximate the loss, we assume to have access to $N=20$ observed value of $\Lcal$.

We consider a ReLU activation function $\sigma(z) = \max\{0,z\}$ and a Gaussian initial condition $\mu_0 \sim \mathcal{N}(1.5,0.1)$. The evolution of $\mu_0$ is obtained with first-order finite volume approximation of the mean-field equation on the domain $[0,3]$ with space grid $\Delta x = 10^{-1}$ and time-step $\Delta t = 10^{-2}$ for the discretization of the time interval $[t_0,T]$ with $t_0=0$ and $T=1$.

We choose as target the distribution obtained at final time $T$ with the value of $(w,b)$ chosen in the center of the parameter grid. This artificially corresponds to a learning problem of the moments of the target distribution.

\begin{figure}[t]
\begin{center}
\includegraphics[width=0.24\linewidth]{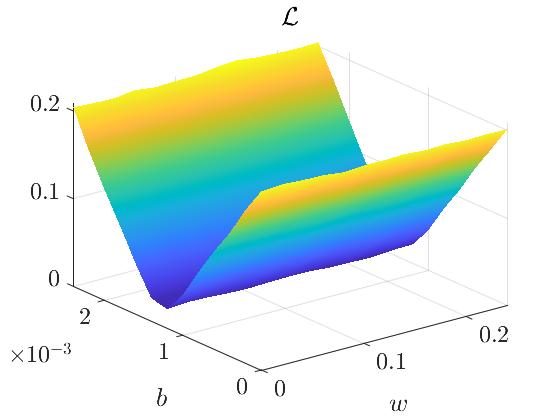}
\includegraphics[width=0.24\linewidth]{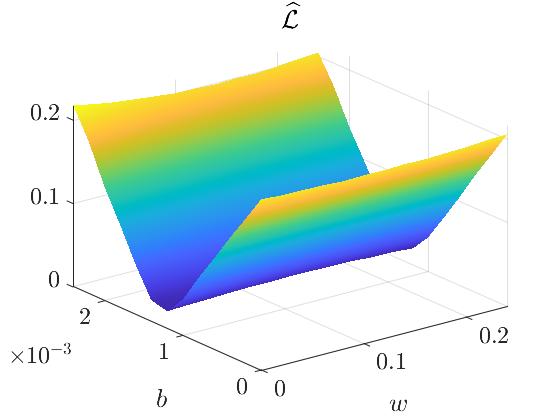}
\includegraphics[width=0.24\linewidth]{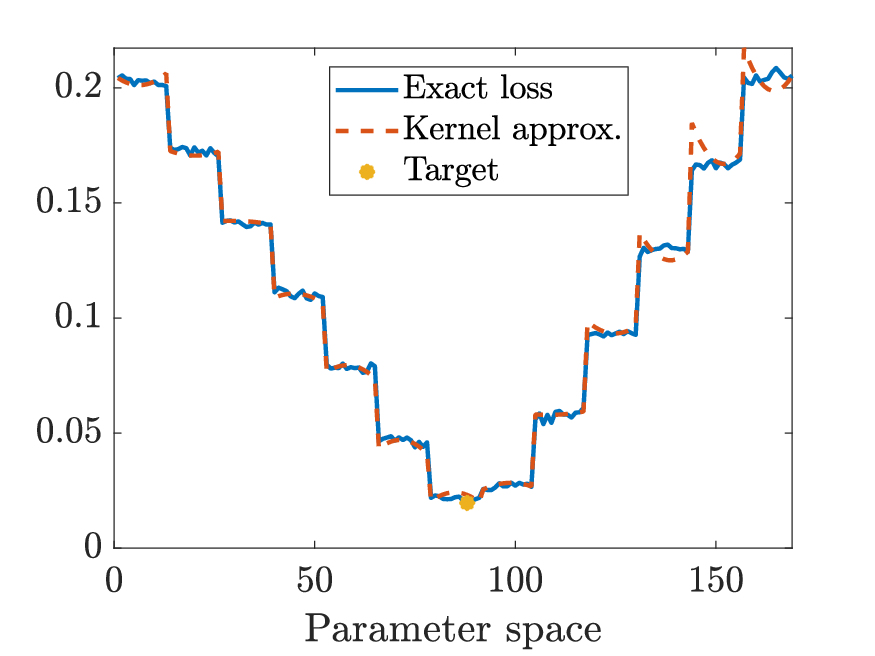}
\includegraphics[width=0.24\linewidth]{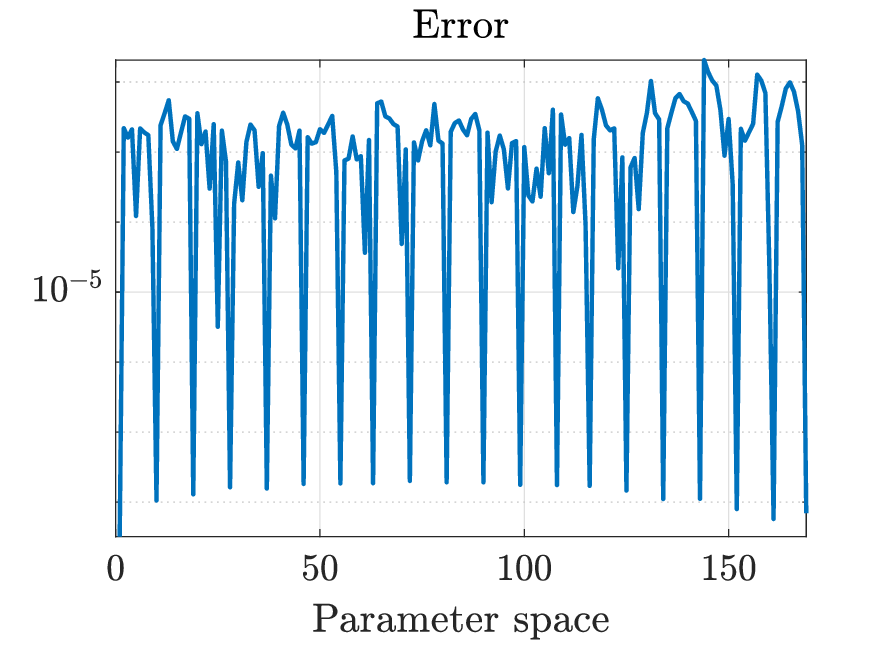}
\caption{Mean-field case. Leftmost panel: exact loss functional computed on the parameter grid. Center-left panel: approximated loss functional via kernel method with $N=20$ points. Center-right panel: comparison between $\mathcal{L}$ and $\widehat{\mathcal{L}}$. Rightmost panel: relative error between $\mathcal{L}$ and $\widehat{\mathcal{L}}$..}
\label{fig:mf}
\end{center}
\end{figure}

In Figure~\ref{fig:mf} we show a comparison between the exact loss functional $\mathcal{L}$ and the approximated one $\widehat{\mathcal{L}}$ via kernel method. In order to compute the distance in~\eqref{eq:mesoL} we have performed a discrete sampling from the target $\nu$ and the output of the network $\mu_T$. The size of the sampling is chosen as $10^2$ and repeated $100$ times for each choice of the parameters in the grid in order to minimize the noise. We observe that  the relative error $(w,b)\mapsto |\mathcal{L}-\widehat{\mathcal{L}}|/|\mathcal{L}|$ is in  the interval $[3\cdot10^{-9},2\cdot10^{-1}]$.

In Figure~\ref{fig:mf2} we observe that $w$ converges towards the optimal parameters during the iteration of the gradient method, and the cost decreases approaching the optimal. Whereas, there is no update of the initial guess of $b$. Moreover, we notice that the target distribution is well approximated.

\begin{figure}[t]
\begin{center}
\includegraphics[width=0.32\linewidth]{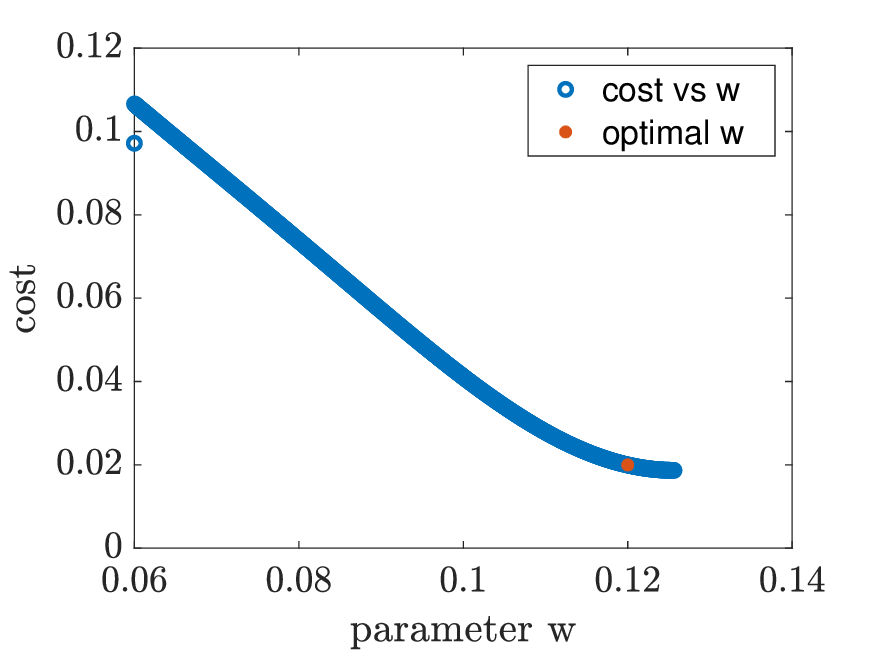}
\includegraphics[width=0.32\linewidth]{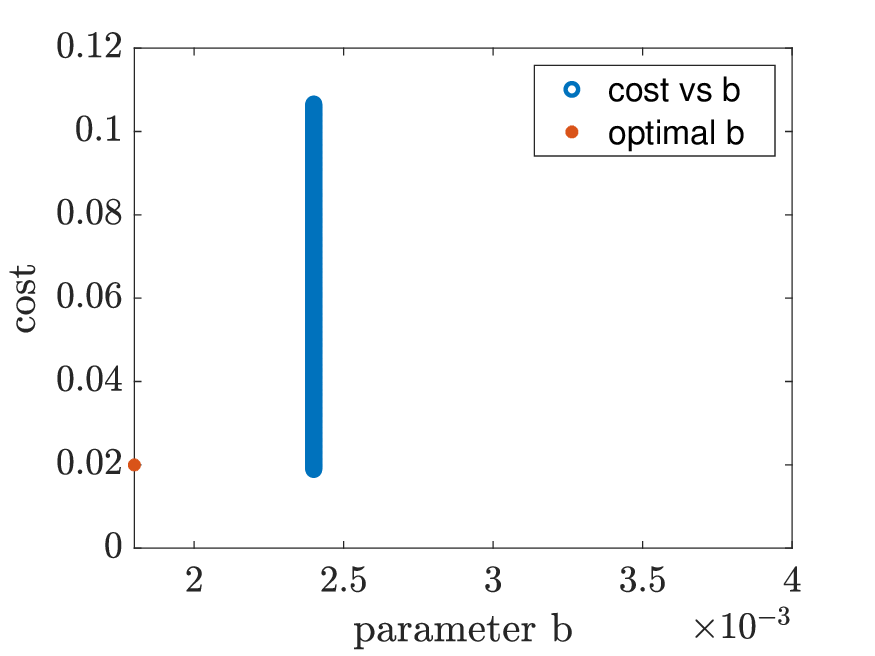}
\includegraphics[width=0.32\linewidth]{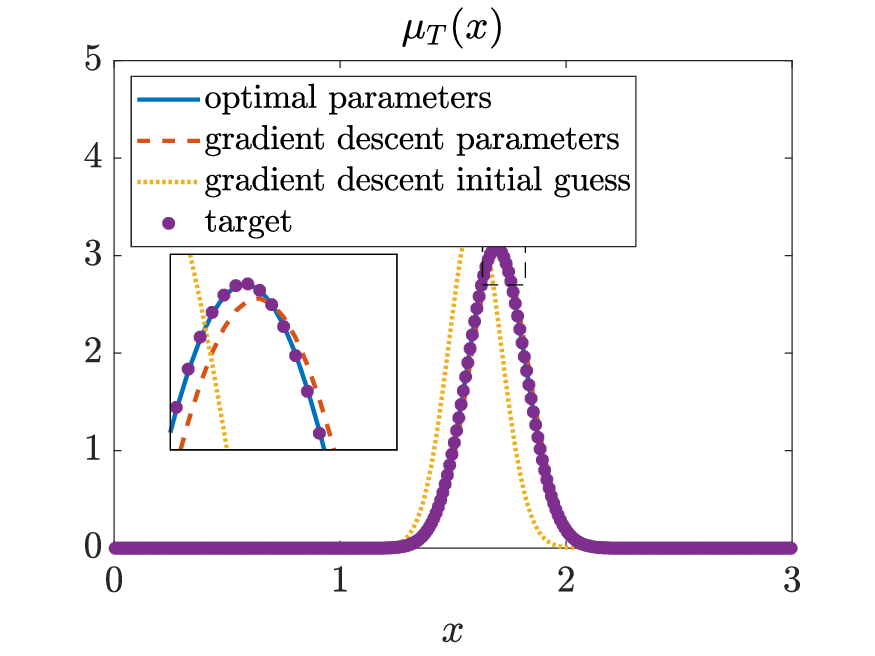}
\caption{Mean-field case. Left: $w$ vs.~the cost during the iteration of the gradient method. Center: $b$ vs.~the cost during the iteration of the gradient method. Right: evolution of the distributions with different parameters.}
\label{fig:mf2}
\end{center}
\end{figure}

\section{Conclusion}

In this chapter, we have considered controllability problems associated with continuous limits of residual neural networks, both from a theoretical and computational point of view. In particular, the contribution of this work is twofold. We have established the exact controllability of linear neural differential equations and  their mean-field reformulation. We have established controllability for bias as control of the system. Then, we  proposed a computational framework to efficiently solve the general optimal control problem associated with the training process of a neural network. The numerical approach is based on the reproducing kernel-based method that allows reformulating the control problem as a reduced unconstrained minimization problem. So far, the technique has been applied to static controls only. Future work includes extending the algorithm to time-dependent control using model predictive control.

\bibliographystyle{abbrv}
\bibliography{biblioCS2,literaturmean,refs}
	
\end{document}